\documentclass[12pt,reqno]{amsart}

\usepackage{amssymb,latexsym}

\usepackage{enumerate}
\allowdisplaybreaks
\usepackage[french,english]{babel}
\usepackage{amsmath}
\usepackage{graphicx}
\usepackage{amssymb}
\usepackage{bbm}
\usepackage{amsthm,mathtools}
\usepackage{ulem}
\usepackage{geometry}
\usepackage{tikz-cd}
\usepackage{mathrsfs}
\usepackage[colorinlistoftodos]{todonotes}
\usepackage{enumitem}
\usepackage{verbatim}
\usepackage[foot]{amsaddr}
\usepackage{dsfont}
\usepackage{cite}
\usepackage[T1]{fontenc}

\makeatletter

\@namedef{subjclassname@2010}{
	
	\textup{2020} Mathematics Subject Classification}

\makeatother
\newtheorem{thm}{Theorem}[section]
\newtheorem*{thm*}{Theorem}

\newtheorem{cor}{Corollary}

\newtheorem{lem}[thm]{Lemma}

\theoremstyle{definition}

\numberwithin{equation}{section}

\newcommand{\bg}{\big}
\newcommand{\bgl}{\big(}
\newcommand{\bgr}{\big)}
\newcommand{\bgg}{\bigg}
\newcommand{\bggl}{\bigg(}
\newcommand{\bggr}{\bigg)}
\newcommand{\Bg}{\Big}
\newcommand{\Bgl}{\Big(}
\newcommand{\Bgr}{\Big)}

\newcommand{\lo}{\log_2}
\newcommand{\lt}{\log_3}
\newcommand{\inv}{^{-1}}
\newcommand{\ph}{\Phi}

\newcommand{\mbr}{\mathbb{R}}

\newcommand{\mbn}{\mathbb{N}}

\newcommand{\mch}{\mathcal{H}}

\newcommand{\mce}{\mathcal{E}}

\newcommand{\mcp}{\mathcal{P}}

\newcommand{\mcm}{\mathcal{M}}

\newcommand{\mmd}{\mathrm{d}}
\newcommand{\mme}{\mathrm{e}}
\newcommand{\mmi}{\mathrm{i}}
\newcommand{\mit}{\mathrm{i}t}
\newcommand{\re}{\operatorname{Re}}

\newcommand{\whp}{\widehat{\Phi}}

\usepackage{hyperref}
\hypersetup{colorlinks=true,linkcolor=blue,anchorcolor=blue,citecolor=blue}
\usepackage{color}
\newcommand{\newabstract}[1]{%
	\par\bigskip
	\csname otherlanguage*\endcsname{#1}%
	\csname captions#1\endcsname
	\item[\hskip\labelsep\scshape\abstractname.]
}

\begin{document}

	\baselineskip=17pt

	\title[Extreme values of the real part of the Riemann zeta function]{Extreme values of the real part of the Riemann zeta function}

	\author{Qiyu Yang\textsuperscript{1}}
    \author{Shengbo Zhao\textsuperscript{2}}
	\address{1.School of Mathematics and Statistics, Henan Normal University, Xinxiang 453007, CHINA}
	\address{2.School of Mathematical Sciences, Key Laboratory of Intelligent Computing and Applications (Tongji University), Ministry of Education, Tongji University, Shanghai 200092, China}
	\email{qyyang.must@gmail.com}
	\email{shengbozhao@hotmail.com}

	\begin{abstract} 
	   In this paper, we establish lower bounds for extreme values of the real part of the Riemann zeta function on the critical line. This work relies on the resonance method of Bondarenko and Seip, together with lower bounds for certain integrals associated with Dirichlet series with non-negative coefficients. Our results extend the work of Bondarenko and Seip (2018).
	\end{abstract}
	
	
	
	\maketitle

\section{Introduction}
\label{secintroduction}

Let \(s=\sigma + \mit\) be a complex variable, and let \(\zeta(s)\) denote the Riemann zeta function. The study of extreme values of \(\zeta(s)\) is meaningful, since it is closely connected with several central questions in analytic number theory, including the distribution of primes,  and the distribution of zeros of \(\zeta(s)\). Soundararajan \cite{soundararajan2008extreme} introduced the resonance method in 2008 and proved
\[
\max_{t \in [T,2T]} \Bg|\zeta\Bgl\frac{1}{2}+\mit \Bgr \Bg| \ge \exp \bggl \bg(1+o(1)\bg) \sqrt{\frac{\log T}{\log\log T}} \bggr
\]
for sufficiently large \(T\). Since then, this method has been widely used in the investigation of a variety of extreme values; see, for example, \cite{aistleitner2019IMRN,Aistleitner2019QJMath,bondarenko2017large,bondarenko2018argument,chirre2019extreme,yang2022extreme,yang2024extreme,qiyu2024JNT} and the references therein.
\par
Based on the idea of Aistleitner \cite{aistleitner2016MathAnn}, Bondarenko and Seip \cite{bondarenko2017large,bondarenko2018argument} developed an improved version of the resonance method and established that, for sufficiently large \(T\),
\[
\max_{t \in [1,T]} \Bg|\zeta\Bgl\frac{1}{2}+\mit \Bgr \Bg| \ge \exp \bggl \bg(1+o(1)\bg) \sqrt{\frac{\log T \log\log\log T}{\log\log T}} \bggr.
\]
Compared with Soundararajan's result, this gives an extra factor of \(\sqrt{\log\log\log T}\) in the exponent of the lower bound. More precisely, Bondarenko and Seip constructed a suitable set on which the relevant GCD sums have sharper lower bounds, and used this set to define their resonator. Their result was further improved by de la Bretèche and Tenenbaum \cite{dlB2019galsum} through optimized estimates for GCD sums, raising the constant from \(\bg(1+o(1)\bg)\) to \(\bgl \sqrt{2}+o(1) \bgr\). This is currently the best known result for extreme values of the Riemann zeta function on the critical line. On the basis of random matrix theory, Farmer, Gonek, and Hughes \cite{Farmer2007JRAM} gave the following prediction
\[
\max_{t \in [0,T]} \Bg|\zeta\Bgl\frac{1}{2}+\mit \Bgr \Bg| \ge \exp \Bgl \bgl \frac{1}{\sqrt{2}}+o(1) \Bgr \sqrt{\log T \log\log T}\Bgr.
\]
\par
Inspired by the earlier work of Levinson \cite{Levinson1973JMAA}, we study extreme values of the real part of the Riemann zeta function on the critical line. By using the resonance method of Bondarenko and Seip \cite{bondarenko2018argument}, we obtain the following result.
\begin{thm}
    \label{thm1}
    Let \(\beta \in [0,1)\) be fixed. Then for sufficiently large \(T\), we have
    \[
    \max_{t \in [T^\beta,T]} \re \zeta  \Bgl\frac{1}{2}+\mit\Bgr \ge \exp\bggl\bgl\sqrt{1-\beta}+o(1)\bgr\sqrt{\frac{\log T \log\log\log T}{\log\log T}} \bggr.
    \]
\end{thm}
\par
Since \(|\zeta(s)| \ge \re \zeta(s)\), the results of Bondarenko and Seip in \cite{bondarenko2017large,bondarenko2018argument} follow immediately from our Theorem \ref{thm1}. This result may be viewed as a directional refinement of usual extreme values of \(\zeta(1/2+\mit)\). Existing lower bounds for the modulus show that \(\zeta(s)\) can be very far from the origin on the critical line, but the modulus itself contains no phase information and therefore does not indicate in which direction such extreme values point in the complex plane. In contrast, extreme values of \(\re\zeta(1/2+\mit)\) show that such largeness can be detected in a prescribed direction. In the present argument this is not achieved by forcing the imaginary part to be small. Rather, the convolution formula transforms \(\zeta(1/2+\mit)\) into a Dirichlet polynomial with non-negative coefficients, and the resonator amplifies the positive real contribution of this polynomial. Heuristically, the positivity of Dirichlet coefficients after convolution may be interpreted as a form of constructive interference; however, the proof does not require showing that all phases of \(\zeta(s)\) are aligned, nor does it imply that the imaginary part is small. Consequently, our result gives simultaneous extreme values of the real part and the modulus of \(\zeta(1/2+\mit)\).
\par
A natural question is whether the constant in Theorem \ref{thm1} can be improved from  \(\bg(1+o(1)\bg)\) to \(\bgl \sqrt{2}+o(1) \bgr\), as in the work of de la Bretèche and Tenenbaum \cite{dlB2019galsum}. This seems difficult using the present method. In \cite{dlB2019galsum}, they obtained a lower bound for 
\[
\max_{t \in [T^\beta,T]}\Bg|\zeta\Bg(\frac{1}{2}+\mit\Bg)\Bg|^2
\]
by using a double-version convolution formula together with optimized estimates for GCD sums. A large lower bound for \(|\zeta(1/2+\mit)|^2\), however, does not guarantee that \(\re \zeta(1/2+\mit)\) is positive and large enough. Indeed, the size of \(|\zeta(1/2+\mit)|^2\) may be dominated by the imaginary part, while the real part may even be negative.
\par
Building on the work of \cite{dong2023Onde, yang2022extreme, YANG2026jmaa}, we can obtain the following two corollaries by arguments similar to that used in the proof of Theorem \ref{thm1}. Corollary \ref{cor1} concerns extreme values of the real part of derivatives of the Riemann zeta function on the critical line.
\begin{cor}
\label{cor1}
    Let \(\ell \in \mbn_+\) and \(\beta \in [0,1)\) be fixed. Then for sufficiently large \(T\), we have
    \[
    \max_{t \in [T^\beta,T]} \re \zeta^{(\ell)} \Bgl\frac{1}{2}+\mit\Bgr \ge \exp\bggl\bgl\sqrt{1-\beta}+o(1)\bgr\sqrt{\frac{\log T \log\log\log T}{\log\log T}} \bggr.
    \]
\end{cor}
\par
Corollary \ref{cor2} deals with extreme values of the real part of derivatives of the Riemann zeta function near the critical line. More precisely, for all sufficiently large \(T\), we assume that \(0<\sigma-1/2 \ll (\log\log T)\inv\).
\begin{cor}
\label{cor2}
    Let \(\ell \in \mbn\) be fixed, and let \(A\) be arbitrary positive number. Then for sufficiently large \(T\), we have
    \[
    \max_{t \in [\sqrt{T},T]} \re \zeta^{(\ell)} \Bgl\frac{1}{2} + \frac{A}{\log\log T} +\mit\Bgr \ge \exp\bggl \bgl \lambda(A)+o(1)\bgr\sqrt{\frac{\log T \log\log\log T}{\log\log T}} \bggr,
    \]
    where
    \[
    \lambda(A) = \frac{1}{\mme^A \sqrt{\mme-1}}.
    \]
\end{cor}
\par 
Since \(\lambda(A) <1 \) for \(A\) sufficiently close to \(0\), Corollary \ref{cor2} implies that near the critical line, the lower bounds for extreme values of the real parts of the Riemann zeta function and its derivatives remain of the same order as \(\sigma=1/2\), while the coefficient in the exponent is slightly smaller.
\par
For simplicity of notation, here and throughout this paper, we write \(\log_j\) for the \(j\)-th iterated logarithm, such as \(\lo x = \log\log x\). Furthermore, we write \(p\) for a prime number, and write 
\[
\widehat{f}(\xi) = \int_\mbr f(x) \mme^{-i x \xi} \mmd x.
\]
for the Fourier transform of \(f \in L^1(\mbr)\).
\par
This paper is organized as follows. In Section \ref{secpre}, we shall give some preliminaries, including the construction of the resonator and several lemmas. In Section \ref{secprove}, we shall prove Theorem \ref{thm1} by the resonance method. In Section \ref{secsketch}, we shall give sketches of the proofs of Corollaries \ref{cor1} and \ref{cor2}.

\section{Preliminaries}
\label{secpre}

\subsection{Constructing the resonator}
\label{secConstucting}

In this subsection, we construct the resonator \(R(t)\) to employ the resonance method. Our construction follows that of Bondarenko and Seip \cite{bondarenko2018argument}.
\par
For large \(T\), let \(N = \lfloor T^\kappa \rfloor\) be a large integer, where \(\kappa \in (0,1)\). Let \(\eta \in (0,1)\) be a parameter to be chosen later. Then, we define 
\[
\mcp = \bg\{p : \mme \log N \log_{2}N < p \leq \exp\bg((\log_2 N)^{\eta}\bg) \log N \log_{2}N \bg\}.
\]
Next, we define the multiplicative function \(f(n)\) to be supported on the set of square-free numbers, with values for \(p \in \mcp\) as
\[
f(p)=\sqrt{\frac{\log N \lt N}{\lo N}} \frac{1}{\sqrt{p}(\log p - \lo N-\lt N)}.
\]
Furthermore, if \(p \notin \mcp\), \(f(p)=0\).
\par
For \(k \in \bg\{1,\dots, \lfloor (\log_2 N)^{\eta} \rfloor\bg\}\), define the set \(\mcp_k\) as follows:
\[
\mathcal{P}_{k} = \bg\{p : \mme^k \log N \log_{2}N < p \leq \mme^{k+1} \log N\log_{2}N \bg\}.
\]
Then for fixed \(b \in (1,1/\eta)\), let
\[
\mathcal{M}_{k} = \left\{ n \in \mathrm{supp}(f) : n \text{ has at least } \Delta_k := \frac{b \log N }{k^2 \log_3 N} \text{ prime divisors in } \mathcal{P}_{k} \right\}.
\]
Let \(\mathcal{M}^{\prime}_{k}\) be the set of integers from \(\mathcal{M}_{k}\) that have prime divisors only in \(\mathcal{P}_{k}\), then set
\[
\mathcal{M} = \mathrm{supp}(f) \setminus \bigcup_{k = 1}^{\lfloor(\log_2 N)^{\eta}\rfloor} \mathcal{M}_{k}.
\]
Clearly, \(\mathcal{M}\) is divisor closed. Indeed, if \(m^{\prime} \mid m \in \mathcal{M}\), we have \(m^{\prime} \in \mathcal{M}\). Following the similar argument as in \cite[Lemma 2]{bondarenko2017large}, we have \(|\mcm| \le N\).
\par
Let \(\mch\) be the set of integers \(h\) such that 
\[
\bgg[\Bgl1+\frac{1}{T} \Bgr^h, \Bgl1+\frac{1}{T} \Bgr^{h+1} \bgg] \bigcap \mcm \neq \emptyset.
\]
Then, for all \(h \in \mch\), let \(m_h\) be the minimum of \(\bg[\bg(1+T\inv\bg)^h, \bg(1+T\inv\bg)^{h+1}\bg] \cap \mcm\), and \(\mcm^\prime\) be the set of all \(m_h\). Define \(r(n)\) to satisfy
\[
r(m_h) = \Bigg( \sum_{n \in \mathcal{M},  (1 +T\inv)^{h - 1} \le n \le (1 + T\inv)^{h + 2}} f(n)^2 \Bigg)^{\frac{1}{2}}
\]
for all \(m_h \in \mcm^\prime\). Here, \(f(n)\) denotes the multiplicative function defined previously. Furthermore, we define the resonator
\[
R(t) = \sum_{m \in \mcm^{\prime}}r(m)m^{-\mit}.
\]
Note that we choose \(N = \lfloor T^\kappa \rfloor\) for some \(\kappa \in (0,1)\). Thus, we have
\begin{align}
\label{R0-upper}
    |R(0)|^2 \le |\mcm^\prime|\sum_{m \in \mcm^\prime} r(m)^2 \ll N \sum_{n \in \mcm}f(n)^2.
\end{align}
\par
Finally, as in \cite{bondarenko2017large, bondarenko2018argument}, we take \(\Phi(t) := \mme^{-t^{2}/2}\) to be a Gaussian function. Then \(\ph(t)\) decays rapidly, and its Fourier transform satisfies \(\whp(\xi) = \sqrt{2\pi}\ph(\xi)>0\). Combining the construction of \(\mcm^\prime\) and \(r(n)\) with the Cauchy-Schwarz inequality, we obtain the following upper bound; see \cite[Lemma 5]{bondarenko2018argument}:
\begin{align}
    \label{R2phi-upper}
    \int_{\mbr}|R(t)|^2\Phi \Big( \frac{t}{T} \Big) \mmd t \ll T \sum_{n \in \mcm}f(n)^2.
\end{align}

\subsection{Auxiliary lemmas}
\label{secauxiliary}

In this subsection, we collect two lemmas that play an important role in the subsequent proof. First, we establish a convolution formula for \(\zeta(s)\). The following Lemma can be found in \cite{bondarenko2018argument}.
\begin{lem}[\cite{bondarenko2018argument}, Lemma 1]
    \label{pro1}
    Let \(\sigma \in [1/2,1)\) be fixed. Assume that \(K(x+\mmi y)\) is a holomorphic
function in the strip \(\sigma-2 \le y \le 0\), satisfying the growth condition
\begin{align}
    \label{pro1-growth}
    \max_{\sigma -2\le y \le 0}|K(x+\mmi y)| = O\Bgl \frac{1}{|x|^2+1} \Bgr
\end{align}
when \(|x| \to \infty\). Then for every \(t\), we have
\begin{align}
    \label{pro1-convolution}
    \int_\mbr \zeta(\sigma+\mmi(t+u))K(u)\mmd u = \sum_{n=1}^\infty \frac{\widehat{K}(\log n)}{n^{\sigma+\mit}}+E(\sigma,t),
\end{align}
where \(E(\sigma,t) = 2\pi K(-t-\mmi(1-\sigma))\).
\end{lem}
\par
The following Lemma shows that the resonator \(R(t)\) constructed in Section \ref{secConstucting} yields an effective lower bound for the contribution from Dirichlet series with non-negative coefficients.
\begin{lem}
    \label{pro2}
    For all \(n \in \mbn_+\), let \(a_n \ge 0\). Assume that \(G(t)=\sum_{n=1}^\infty a_n n^{-1/2-\mit}\) is absolutely convergent. Let \(\varepsilon\) be a small positive number. Then for sufficiently large \(T\), we have
    \begin{align*}
        \int_\mbr\re G(t)& |R(t)|^2\ph\Bgl\frac{t}{T} \Bgr \mmd t  \nonumber \\
        &\ge T \bgl\min_{n \le T^\varepsilon} a_n\bgr \exp \bggl \bgl \eta \sqrt{\kappa} +o(1) \bgr \sqrt{\frac{\log T \lt T}{\lo T}}\bggr \sum_{n \in \mcm}f(n)^2.
    \end{align*}
\end{lem}
\begin{proof}
    Taking the real part of both sides of \cite[Lemma 6, Eq. (23)]{bondarenko2018argument} suffices; we refer to \cite[p. 1699]{bondarenko2017large} for further details.
\end{proof}

\section{Proof of Theorem \ref{thm1}}
\label{secprove}

In this section, we prove Theorem \ref{thm1}. To this end, for small \(\varepsilon>0\), we set 
\[
K(t) := \frac{\sin^2((\varepsilon \log T)t)}{(\varepsilon \log T)t^2}.
\]
as in \cite{bondarenko2018argument}. It is easy to see that Lemma \ref{pro1} holds for such \(K(t)\). In addition, its Fourier transform satisfies 
\begin{align}
    \label{Kfourier-lower}
    \widehat{K}(\xi) =\pi \max \Bgl \Bgl 1- \frac{|\xi|}{2\varepsilon \log T},0 \Bgr \Bgr.
\end{align}
Then, we consider the following integral
\[
\int_{\mbr^2} \re \zeta\Bgl \frac{1}{2} +\mmi (t+u) \Bgr K(u) |R(t)|^2 \ph \Bgl \frac{t}{T}\Bgr \mmd u\mmd t.
\]
For brevity, we denote the integrand by \(g(t,u)\), that is,
\[
g(t,u) := \re \zeta\Bgl \frac{1}{2} +\mmi (t+u) \Bgr K(u) |R(t)|^2 \ph \Bgl \frac{t}{T}\Bgr.
\]
\par
First, we have
\begin{align}
\label{intsmall-1}
    \bgg|\int_{|t| \le T^\beta} & \int_\mbr \re \zeta\Bgl \frac{1}{2} +\mmi (t+u) \Bgr K(u) \mmd u \mmd t \bgg| \nonumber \\
    &  \le  \Bgl \int_{|t| \le T^\beta}  \int_{|u| \le T^\beta}+ \int_{|t| \le T^\beta}  \int_{|u| > T^\beta} \Bgr \Bg|\zeta\Bgl \frac{1}{2}
    + \mmi (t+u) \Bgr K(u)\Bg|  \mmd u \mmd t. 
\end{align}
Combining \eqref{pro1-growth} with the following classical convexity bound\footnote{Sharper upper bounds can be found in \cite{Bourgain2017JAMS}. However, \eqref{convexitybound} is already sufficient for our purpose.}
\begin{align}
    \label{convexitybound}
\Bg| \zeta\Bgl \frac{1}{2}+\mit \Bgr \Bg| \ll (1+|t|)^\frac{1}{6},
\end{align}
the second term on the right-hand side of \eqref{intsmall-1} can be bounded by \(T^\beta\). Furthermore, the second moment of \(\zeta(1/2+\mit)\) admits the following upper bound:
\[
\int_{|t| \le T} \Bg|\zeta\Bgl \frac{1}{2}+ \mit\Bgr \Bg|^2 \mmd t \ll T \log T.
\]
Thus, the Cauchy-Schwarz inequality yields that the first term on the right-hand side of \eqref{intsmall-1} can be bounded by \(T^\beta \sqrt{\log T}\). Substituting the two upper bounds above into \eqref{intsmall-1} and combining with \eqref{R0-upper}, we obtain
\begin{align}
    \label{intsmallR-1}
    \bgg|\int_{|t| \le T^\beta} \int_\mbr g(t,u) \mmd u \mmd t \bgg| \ll T^\beta \sqrt{\log T} R(0)^2\ll T^{\beta+\kappa}\sqrt{\log T}\sum_{n \in \mcm}f(n)^2.  
\end{align}
\par
Next, we consider the tail integral over the range \(|t| > T\log T\). The definition of \(\Phi(t)\) and \eqref{R0-upper} yield
\begin{align*}
    \bgg|\int_{|t|> T \log T}&  \int_\mbr g(t,u) \mmd u \mmd t \bgg|  
    \le \int_{|t| > T \log T} \int_\mbr \Bg|\zeta\Bgl \frac{1}{2}
    + \mmi (t+u) \Bgr K(u)\Bg| |R(t)|^2 \ph \Bgl\frac{t}{T}\Bgr \mmd u \mmd t \\
    & \ll T^\kappa \mme^{-\frac{(\log T)^2}{4}} \sum_{n \in \mcm}f(n)^2  \int_{|t|>T \log T} \int_\mbr \Bg|\zeta\Bgl \frac{1}{2}+ \mmi (t+u) \Bgr K(u)\Bg| \Bgl\frac{t}{T}\Bgr \mmd u \mmd t.
\end{align*}
Due to the rapid decay of the exponential function, it follows that
\begin{align}
        \label{intsmallR-2}
      \bgg|\int_{|t| > T\log T} \int_\mbr g(t,u) \mmd u \mmd t \bgg| = o \Bgl\sum_{n \in \mcm}f(n)^2 \Bgr.
\end{align}
Thus, combining \eqref{intsmallR-1} and \eqref{intsmallR-2}, we obtain
\begin{align}
\label{intR-asy}
    \int_{T^\beta \le |t| \le T\log T} \int_\mbr g(t,u)\mmd u \mmd t 
    = \int_{\mbr^2} g(t,u)\mmd u \mmd t + O\Bgl T^{\beta+\kappa}\sqrt{\log T}\sum_{n \in \mcm}f(n)^2 \Bgr.
\end{align}
\par
For the double integral over the range \(T^\beta \le |t| \le T\log T\) and \(u \in \mbr\),  we have
\begin{align}
\label{intbigR-1}
& \int_{T^\beta \le |t| \le T\log T} \int_\mbr g(t,u)\mmd u \mmd t \nonumber \\
&= \int_{T^\beta \le |t| \le T\log T}  \Bgl \int_{T^\beta/2 \le |t+u| \le 2T\log T}+\int_{|u+t| < T^\beta/2}+\int_{|u+t| > 2T\log T} \Bgr  g(t,u) \mmd u \mmd t \nonumber \\
&= : \int_{T^\beta \le |t| \le T\log T}  \int_{T^\beta/2 \le |t+u| \le 2T\log T}g(t,u) \mmd u \mmd t+I_1+I_2.
\end{align}
Set \(I_0 : =I_1+I_2\). Let \(v=u-t\), which yields
\begin{align*}
    |I_0| & \ll \int_{T^\beta \le |t| \le T\log T} \int_{\{|v| < T^\beta /2\}\cup\{|v| > 2T\log T\}}\Bg|\zeta\Bgl \frac{1}{2}
    + \mmi v \Bgr K(v-t)\Bg| |R(t)|^2 \ph \Bgl\frac{t}{T}\Bgr \mmd v \mmd t \\
    & \le \int_{T^\beta \le |t| \le T\log T} \int_{\{|v| < T^\beta /2\}\cup\{|v| > 2T\log T\}}\Bg|\zeta\Bgl \frac{1}{2}
    + \mmi v \Bgr K\Bg(\frac{v}{2}\Bg)\Bg| |R(t)|^2 \ph \Bgl\frac{t}{T}\Bgr \mmd v \mmd t. 
\end{align*}
Combining \eqref{pro1-growth} with \eqref{convexitybound}, we get the following upper bound for \(I_0\):
\begin{align}
\label{intI0-upper}
    |I_0| \ll \int_{T^\beta \le |t| \le T\log T}|R(t)|^2 \ph \Bgl\frac{t}{T}\Bgr \mmd t \ll T \sum_{n \in \mcm}f(n)^2.
\end{align}
Here, in the last step, we use \eqref{R2phi-upper}. Substituting \eqref{intI0-upper} into \eqref{intbigR-1}, we obtain
\begin{align}
    \label{intbig-asy}
    \int_{T^\beta \le |t| \le T\log T}& \int_\mbr g(t,u)\mmd u \mmd t \nonumber \\ &=\int_{T^\beta \le |t| \le T\log T} \int_{T^\beta/2 \le |t+u| \le 2T\log T} g(t,u)\mmd u \mmd t + O\Bgl T \sum_{n \in \mcm}f(n)^2 \Bgr.
\end{align}
\par
Let \(\kappa\) satisfy \(\beta+\kappa<1\). From \eqref{intR-asy} and \eqref{intbig-asy}, it follows that
\begin{align}
\label{int-asy}
    \int_{T^\beta \le |t| \le T\log T} \int_{T^\beta/2 \le |t+u| \le 2T\log T}& g(t,u)\mmd u \mmd t \nonumber \\
    &= \int_{\mbr^2}  g(t,u)\mmd u \mmd t + O\Bgl T \sum_{n \in \mcm}f(n)^2 \Bgr.
\end{align}
Applying \eqref{R2phi-upper} to the left-hand side of \eqref{int-asy}, we have 
\begin{align}
\label{max-gg}
    \Bgl\max_{t \in [T^\beta/2,2T\log T]} \re \zeta\Bgl\frac{1}{2}+\mit\Bgr \Bgr&  T  \sum_{n \in \mcm}f(n)^2 \nonumber \\
    & \gg \int_{\mbr^2}  g(t,u)\mmd u \mmd t + O\Bgl T \sum_{n \in \mcm}f(n)^2 \Bgr.
\end{align}
\par
Set \(\sigma = 1/2\) in Lemma \ref{pro1}, take the real part of both sides of \eqref{pro1-convolution}, then integrate over \(t \in \mbr\) to obtain
\begin{align}
\label{intRR-equ}
    \int_{\mbr^2}   g(t,u)\mmd u \mmd t 
    = \int_\mbr \re \sum_{n=1}^\infty \frac{\widehat{K}(\log n)}{n^{1/2+\mit}} |R(t)|^2 \ph \Bgl\frac{t}{T}\Bgr  \mmd t + \mce(R,T),
\end{align}
where
\[
\mce(R,T) := \int_\mbr \re E\Bgl \frac{1}{2}, t \Bgr |R(t)|^2 \ph \Bgl\frac{t}{T}\Bgr  \mmd t
\]
For the second term on the right-hand side of \eqref{intRR-equ}, we derive the following crude upper bound by the definition of \(K(u)\) together with \eqref{R0-upper}
\begin{align}
    \label{E-upper}
    |\mce(R,T)|\ll T^{\kappa+\varepsilon} \sum_{n \in \mcm}f(n)^2,
\end{align}
where \(\varepsilon>0\) is the parameter defining the function \(K(u)\). Then, set
\[
G(t) := \sum_{n=1}^\infty \frac{\widehat{K}(\log n)}{n^{1/2+\mit}}.
\]
Lemma \ref{pro2} shows that 
\begin{align*}
     \int_\mbr\re G(t)& |R(t)|^2\ph\Bgl\frac{t}{T} \Bgr \mmd t  \nonumber \\
        &\ge T \bgl\min_{n \le T^\varepsilon} \widehat{K}(\log n)\bgr \exp \bggl \bgl \eta \sqrt{\kappa} +o(1) \bgr \sqrt{\frac{\log T \lt T}{\lo T}}\bggr \sum_{n \in \mcm}f(n)^2.
\end{align*}
By \eqref{Kfourier-lower}, \(\widehat{K}(\log n) \ge \pi/2\) holds for \(n \le T^\varepsilon\). Substituting into the above formula, we obtain 
\begin{align}
\label{G-lower}
         \int_\mbr\re G(t) |R(t)|^2\ph\Bgl\frac{t}{T} \Bgr \mmd t \gg T  \exp \bggl \bgl \eta \sqrt{\kappa} +o(1) \bgr \sqrt{\frac{\log T \lt T}{\lo T}}\bggr \sum_{n \in \mcm}f(n)^2.
\end{align}
Plugging \eqref{E-upper} and \eqref{G-lower} into \eqref{intRR-equ} yields that 
\begin{align*}
    \int_{\mbr^2} &g(t,u) \mmd u  \mmd t\\
   & \gg  T  \exp \bggl \bgl \eta \sqrt{\kappa} +o(1) \bgr \sqrt{\frac{\log T \lt T}{\lo T}}\bggr   \sum_{n \in \mcm}f(n)^2 
     + O\Bgl T^{\kappa+\varepsilon} \sum_{n \in \mcm}f(n)^2\Bgr .
\end{align*}
Let \(\eta\), \(\kappa\) be taken sufficiently close to \(1\) and \(1-\beta\), respectively. Combining with \eqref{max-gg}, we deduce that
\begin{align*}
    \max_{t \in [T^\beta/2,2T\log T]} \re \zeta\Bgl\frac{1}{2}+\mit\Bgr  \ge \exp\bggl \bgl \sqrt{1-\beta} +o(1)\bgr \sqrt{\frac{\log T \lt T}{\lo T}} \bggr.
\end{align*}
Finally, we set 
\[
T^\prime = \frac{T}{2\log T}
\]
and readjust the parameter \(T=T^\prime\) following the idea from \cite{bondarenko2018argument}. Variations in the logarithmic factor affect only the lower order terms in the exponent on the right-hand side of the above formula. Hence we obtain
\[
\max_{t \in [T^\beta,T]} \re \zeta\Bgl\frac{1}{2}+\mit\Bgr \ge \exp\bggl\bgl\sqrt{1-\beta}+o(1)\bgr\sqrt{\frac{\log T \lt T}{\lo T}} \bggr
\]
for sufficiently large \(T\), which completes the proof of Theorem \ref{thm1}.

\section{Sketches of the proofs of Corollaries \ref{cor1} and \ref{cor2}}
\label{secsketch}

For the proof of Corollary \ref{cor1}, we replace the function \(G(t)\) in Section \ref{secprove} by
\[
G(t) := \sum_{n=1}^\infty \frac{\widehat{K}(\log n)(\log n)^\ell}{n^{1/2+\mit}}.
\]
Then, we apply the convexity bound for derivatives of the Riemann zeta function on the critical line, as given in \cite[Lemma 2]{yang2022extreme}, together with the upper bound for the moments in \cite[Theorem A\textsuperscript{\(\prime\prime\)}]{Ingham1926PLMS}. Furthermore, to simplify the computation, one may replace \(\Phi(t/T)\) by \(\Phi(t\log T/T)\).
\par
For the proof of Corollary \ref{cor2}, we slightly modify the construction of the resonator, as in \cite[Section 3.1]{YANG2026jmaa}, so that it applies to the case where \(\sigma\) is close to the critical line. Moreover, the function \(G(t)\) is defined as above. Combining this modified resonator with the upper bound in \cite[Lemma 4]{yang2022extreme} for the moments of derivatives of the Riemann zeta function in the critical strip, also see \cite{Ingham1926PLMS}, gives Corollary \ref{cor2}. Similarly to Corollary \ref{cor1}, we may take \(\Phi(t/T)\) to \(\Phi(t\log T/T)\) for ease of computation.

\section*{Acknowledgments}
Qiyu Yang was supported by the Natural Science Foundation of Henan Province (Grant No. 252300421782).
	
	\bibliographystyle{siam}
    \bibliography{reference}

@article{aistleitner2016MathAnn,
  author = {C. Aistleitner},
  title = {Lower bounds for the maximum of the {R}iemann zeta function along vertical lines},
  journal = {Math. Ann.},
  volume = {365},
  number = {1-2},
  pages = {473--496},
  year = {2016}
}

@article{aistleitner2019IMRN,
  title={Extreme values of the {R}iemann zeta function on the 1-line},
  author={Aistleitner, Christoph and Mahatab, Kamalakshya and Munsch, Marc},
  journal={Int. Math. Res. Not.},
  volume={IMRN 2019},
  number={22},
  pages={6924--6932},
  year={2019}
}

@article{Aistleitner2019QJMath,
  author = {C. Aistleitner and K. Mahatab and M. Munsch and A. Peyrot},
  title = {On large values of \({L}(\sigma, \chi)\)},
  journal = {Q. J. Math.},
  volume = {70},
  number = {3},
  pages = {831--848},
  year = {2019}
}

@article{bondarenko2017large,
  author = {A. Bondarenko and K. Seip},
  title = {Large greatest common divisor sums and extreme values of the {R}iemann zeta function},
  journal = {Duke Math. J.},
  volume = {166},
  number = {9},
  pages = {1685--1701},
  year = {2017}
}

@article{bondarenko2018argument,
    author = {Bondarenko, A and Seip, K},
    title = {Extreme values of the {R}iemann zeta function and its argument},
    journal = {Math. Ann.},
    volume = {372},
    number = {3-4},
    pages = {999--1015},
    year = {2018}
}

@article{Bourgain2017JAMS,
    author = {Bourgain, J},
    title = {Decoupling, exponential sums and the {R}iemann zeta function},
    journal = {J. Amer. Math. Soc.},
    year = {2017},
    volume = {30},
    pages = {205--224}
}

@article{chirre2019extreme,
  title={Extreme values for ${S}_n (\sigma, t)$ near the critical line},
  author={Chirre, Andr{\'e}s},
  journal={J. Number Theory},
  volume={200},
  pages={329--352},
  year={2019}
}

@article{dlB2019galsum,
    author = {R. de la Bret{\`e}che and G. Tenenbaum},
    title = {Sommes de {G}\'al et applications},
    journal = {Proc. Lond. Math. Soc.},
    volume={119},
    pages={104--134},
    year = {2019}
}

@article{dong2023Onde,
  title = {On derivatives of zeta and {$L$}-functions},
  author = {Zikang Dong and Yutong Song and Weijia Wang and Hao Zhang},
  journal = {Ramanujan J.},
  volume = {66},
  number = {1},  year={2025},
  pages = {5-21}
}

@article{Farmer2007JRAM,
    author = {D. W. Farmer and S. M. Gonek and C. P. Hughes},
    title = {The maximum size of {$L$}-functions},
    journal = {J. Reine Angew. Math.},
    year = {2007},
    pages = {215--236},
    volume = {609}
}

@article{Ingham1926PLMS,
    author = {A. E. Ingham},
    title = {Mean-value theorems in the theory of the {R}iemann zeta-function},
    journal = {Proc. Lond. Math. Soc.},
    year = {1926},
    volume = {27},
    page = {273--300}
}

@article{Levinson1973JMAA,
title = {{$\Omega$}-{T}heorems for the real part of the {R}iemann zeta function},
journal = {J. Math. Anal. Appl.},
volume = {43},
number = {1},
pages = {123-127},
year = {1973},
author = {Norman Levinson}
}

@article{soundararajan2008extreme,
  author = {K. Soundararajan},
  title = {Extreme values of zeta and {$L$}-functions},
  journal = {Math. Ann.},
  volume = {342},
  number = {2},
  pages = {467--486},
  year = {2008}
}

@article{yang2022extreme,
  author = {D. Yang},
  title = {Extreme values of derivatives of the {R}iemann zeta function},
  journal = {Mathematika},
  volume = {68},
  number = {2},
  pages = {486-510},
  year = {2022}
}

@article{yang2024extreme,
  title={Extreme values of derivatives of zeta and {$L$}-functions},
  author={Daodao Yang},
  journal={Bull. Lond. Math. Soc.},
  volume={56},
  number={1},
  pages={79--95},
  year={2024}
}

@article{qiyu2024JNT,
    title = {Large values of $\zeta(s)$ for $1/2<${R}e$(s)<1$},
journal = {J. Number Theory},
volume = {254},
pages = {199-213},
year = {2024},
author ={Qiyu Yang}
}

@article{YANG2026jmaa,
title = {Large values of derivatives of the {R}iemann zeta function on vertical homogeneous progressions},
journal = {J. Math. Anal. Appl.},
volume = {560},
note  = {Paper No. 130487, 12 pp.},
year = {2026},
author = {Qiyu Yang and Shengbo Zhao},
}
\end{document}